\documentclass[12pt]{amsart}
\usepackage{amsmath}

\def\be{\begin{equation}}
\def\ee{\end{equation}}

\newtheorem{thm}{Theorem}[section]

\newtheorem{lem}{Lemma}[section]
\newtheorem{prop}{Proposition}[section]

\theoremstyle{definition}

\newtheorem{example}{Example}[section]
\theoremstyle{remark}

\numberwithin{equation}{section}

\newcommand{\rnn}{\mathbb{R}^{2n}}
\newcommand{\cn}{\mathbb{C}^{n}}

\newcommand{\ci}{C^\infty}

\newcommand{\Lap}{\overline \Delta}

\begin{document}

\title{Lagrangian angles of family of Lagrangian fibrations under mean curvature flow}
\author{John Man-shun  Ma, Tom Yau-heng Wan}

\address{Department of Mathematics,
The Chinese University of Hong Kong,
Shatin, Hong Kong, P. R. China.} \email{msma@math.cuhk.edu.hk}

\address{Department of Mathematics,
The Chinese University of Hong Kong,
Shatin, Hong Kong, P. R. China.} \email{tomwan@math.cuhk.edu.hk}

%

\begin{abstract}
In this paper, we discuss the Lagrangian angles of a family of Lagrangian fibrations moved under mean curvature flow. In the case of one complex dimension, the angle function is shown to satisfy a degenerated partial differential equation. We prove that any smooth solution to the equation also corresponds to a smooth foliation of curves under mean curvature flow.
\end{abstract}
\maketitle
\markboth{}{}

\section{Introduction}
In this paper we study Lagrangian mean curvature flow. The motivation comes from the SYZ conjecture \cite {SYZ}, which postulates the existence of special Lagrangian tori and dual tori fibration on a Calabi-Yau manifold $X$ and its mirror $\hat X$ respectively. To prove the conjecture, a difficulty is to construct the fibration. As special Lagrangian submanifold is minimal, one hope that the mean curvature flow will move a Lagrangian submanifold to a special one.\\

Many results about mean curvature flow of a single Lagrangian submanifold have been derived in the last decade \cite{S}, \cite{SW}, \cite{TY}. In this paper, instead of a single Lagrangian submanifold, we consider a family of Lagrangian submanifolds which foliates a domain in $\cn$ and is moved together under the mean curvature flow. The idea was suggested by S.T. Yau and Conan Leung. Doing this might cause more complications, but the advantage is that we can consider the Lagrangian angles as a function in the ambient space. In the case of complex dimension one, the Lagrangian angle function is shown to satisfy a nonlinear second order degenerated partial differential equation. Conversely, we prove that any smooth solution of the PDE corresponds to a solution to mean curvature flow of a smooth foliation.\\

The paper is organized as follows. In section 2 we state the basic facts concerning Lagrangian submanifold and mean curvature flow that we need. In section 3 we consider the foliation of an open set $W$ in $\mathbb C$ by Lagrangian submanifolds and derive equation for its Lagrangian angle. In section 4 we prove the main theorem of this paper. In section 5 we consider examples of invariant solutions.

\section{Basic notations in Riemannian geometry and mean curvature flow in $\mathbb{C} ^n$}

First we review some notations and results used in this article. Let $(L^n,g)$ and $(M^{n+m}, \bar{g})$ be Riemannian manifolds. We write $\nabla$ and $\overline{\nabla}$ to denote the Levi-Civita connection of $L$ and $M$ respectively. Let
\begin{equation*}
F: L \to M
\end{equation*}
be an isometric embedding, which is the same as saying that $\nabla =(\overline \nabla)^\top$ if we consider $L\subset M$ as a submanifold. Let
\begin{equation*}
\begin{split}
A: TL\times TL &\longrightarrow NL\\
(X, Y)& \mapsto (\overline {\nabla}_X Y)^\perp
\end{split}
\end{equation*}
be the second fundamental form and
\begin{equation}
H= trA
\end{equation}
be the mean curvature vector.

Given a $\ci$ function $h: M \to \mathbb R$, the Laplacian of $h$, denoted $\overline \Delta h$, is defined as
\begin{equation*}
\overline \Delta h = tr \overline \nabla ^2 h \ .
\end{equation*}
By composition with $F: L \to M$, we get a function $F^*h :L\to \mathbb R$. The two Laplacians $\Delta F^*h$ and $\overline \Delta h$ are related as follows: Let $x\in L$ and $e_1, ..., e_n$, $f_1, ..., f_m$ be orthonormal frame in $T_xL$ and $NL_x$ respectively, then

\begin{equation}
\begin{split}
\Lap h &= tr \overline \nabla ^2 h \\
&= \sum _{i=1}^n \overline \nabla ^2h (e_i, e_i) + \sum _{j=1}^m \overline \nabla ^2h (f_j, f_j)\\
&= \sum _{i=1}^n (h_{ii}- \overline \nabla _{\overline \nabla _{e_i} e_i} h) + \sum _{j=1}^m \overline \nabla ^2h (f_j, f_j)\\
&= \sum _{i=1}^n (h_{ii}- \nabla _{\nabla _{e_i} e_i} h -\overline \nabla _{(\overline {\nabla}_{e_i} e_i) ^\perp} h) + \sum _{j=1}^m \overline \nabla ^2h (f_j, f_j)\\
&= \Delta F^*h - \overline \nabla _H h+ \sum _{j=1}^m \overline \nabla ^2h (f_j, f_j) \ .\\
\end{split}
\end{equation}

Now consider the case $M= \cn \cong \rnn$ equipped with the standard euclidean metric $\bar g$. On $\rnn$ we denote $y^{\alpha}, \alpha =1, ..., 2n$, the standard coordinates, so the complex structure $J$ is given by
\begin{equation*}
\begin{split}
J(\frac {\partial}{\partial y^i})&= \frac {\partial}{\partial y^{n+i}} \ ,\\
J(\frac {\partial}{\partial y^{n+i}})&= -\frac {\partial}{\partial y^i} \ .
\end{split}
\end{equation*}

The symplectic form $w$ is defined as $w(X,Y)=\bar g (JX,Y)$. A submanifold $F: L\hookrightarrow \cn$ is called a Lagrangian submanifold, if dim$L = n$ and the restriction of the symplectic form to $L$ is zero, i.e.
\[F^* w =0 \ .\]
Note that the definition implies that $J$ will send tangent vector of $L$ to its normal bundle.

Now consider the holomorphic $(n,0)$-form $dz=dz^1\wedge ... \wedge dz^n$ defined on $\cn$, where $dz^i=dy^i+\sqrt{-1}dy^{n+i}$. It is well known that if $L$ is a Lagrangian submanifold, then
\[
dz|_{T_xL}=e^{i\theta (x)} dvol_L
\]
for a $\ci$ function $\theta :L\to \mathbb R /2\pi \mathbb Z$. The function $\theta$ is called the {\bf Lagrangian angle} of $L$. If $n=1$, the Lagrangian angle is just the angle which the tangent vector makes with the $x$-axis. By differentiating the above equation, we obtain
\begin{equation}
d\theta = \iota _H w \ .
\end{equation}

Now we consider the mean curvature flow. The mean curvature flow of $F_0:L \to M$ is a $\ci$ map $$F: L \times [0, T) \to M$$ which satisfies
\begin{equation*}
\begin{cases}
\frac{\partial}{\partial t}F(x,t)=H(x,t)\\
F(x,0)=F_0(x) \ .
\end{cases}
\end{equation*}

In \cite{S}, the author studied the mean curvature flow of Lagrangian submanifold $L$ in K\"{a}hler manifold $(M,w)$. In general, the mean curvature flow might not preserve the Lagrangian condition. A necessary condition is that $d\iota _Hw=0$, as closed form on $L$ corresponds to infinitesimal Lagrangian deformation. The condition can be satisfied, for example, in the case that $M$ is K\"{a}hler Einstein. In the case $M$ is Calabi-Yau, for each $t$ the submanifold $L_{t}:=F(\cdot,t)$ is thus Lagrangian, so the Lagrangian angle is a well-defined map
\[\theta: L \times [0,T) \longrightarrow \mathbb R /2\pi \mathbb Z \ .\]
It was proven in \cite{TY} that in this case the Lagrangian angle satisfies the following ``heat equation"
\begin{equation}
\frac{\partial \theta}{\partial t} = \Delta_{L_t} \theta \ .
\end{equation}

\section{Lagrangian Angle of a Foliation}
In this section, we consider a $\ci $ map $$F: L \times U \times [0,T) \longrightarrow \cn \ ,$$
where $L$ is a $n$ dimensional manifold, and $U$ is an open set in $\mathbb{R}^n$ such that:
\begin{enumerate}
\item For each fixed $u=(u_1,...,u_n)$ and $t$, $F(\cdot, u, t): L \longrightarrow \cn$ is a embedding of Lagrangian submanifold.
\item For each fixed $t$, the map $F_t :=F(\cdot, \cdot, t): L \times U \longrightarrow \cn$ is a diffeomorphism onto an open set $W_t$ in $\cn$.
\item For each fixed $u=(u_1,...,u_n)$, $F(\cdot, u, \cdot): L \times [0,T) \longrightarrow \cn$ is a solution to the mean curvature flow, i.e. $$\frac{\partial}{\partial t} F(x,u,t)= H(x,u,t) \ ,$$
\end{enumerate}
where $H(x,u,t)=H(F(x,u,t))$ is the mean curvature vector to $F(\cdot ,u, t)$ at $x$.

Intuitively we are considering a foliation of $W_0 \subset \cn$ by Lagrangian submanifolds, and each of the submanifolds are then moved by mean curvature flow. Assuming that the solution at time $t$ is still a smooth foliation of an open set $W_t$, the Lagrangian angle can be regarded as a smooth function
\[\theta: L \times U \times [0,T) \longrightarrow \mathbb R / 2\pi \mathbb Z \ .\]
By condition (2) we can pullback the Lagrangian angle $\theta$ to a function on $W=\{(y,t): y\in W_t\}$. More precisely, we define
\begin{equation*}
\begin{split}
\tilde{\theta} : W & \longrightarrow \mathbb R /2\pi \mathbb Z\\
\tilde{\theta}(y,t) & = \theta (F_t^{-1}(y),t) \ .
\end{split}
\end{equation*}

Now we derive a equation for $\tilde{\theta}$. Let $(y,t)\in W$ be given. Then $y=F_t(x,u)$ for a unique $(x,u)\in L \times U$. Let $e_1, ..., e_n$ be orthonormal basis for the vector space $(F_t)_*T_xL$. The Lagrangian Condition (1) implies that $Je_1, ..., Je_n$ forms a orthonormal basis on $NL_x$. Thus by (2.2) and (2.4) we have
\begin{equation}
\begin{split}
\Lap \tilde{\theta} &= \Delta _{L_t} F_t^*\tilde{\theta}- \overline \nabla _H\tilde{\theta}+ \sum _{i=1}^n \overline \nabla ^2\tilde{\theta} (Je_i, Je_i)\\
&= \Delta _{L_t}\theta- \overline \nabla _H \tilde{\theta}+ \sum _{i=1}^n \overline \nabla ^2\tilde{\theta} (Je_i, Je_i)\\
&= \frac{\partial \theta}{\partial t}- \overline \nabla _H \tilde{\theta}+ \sum _{i=1}^n \overline \nabla ^2\tilde{\theta} (Je_i, Je_i) \ .\\
\end{split}
\end{equation}

By definition, we have $$\tilde\theta (y,t)=\tilde \theta (F_t(x,u), t)=\theta (x,u,t) \ ,$$
hence
$$\overline \nabla_{\frac{\partial F_t}{\partial t}}\tilde \theta + \frac{\partial \tilde\theta}{\partial t}
= \frac{\partial \theta}{\partial t} \ .$$

As $\frac{\partial F_t}{\partial t} = H$, (3.1) implies that the phase angle satisfy the following differential equation
\begin{equation*}
\Lap \tilde{\theta} = \frac{\partial \tilde{\theta}}{\partial t} + \sum _{i=1}^n \overline \nabla ^2\tilde{\theta} (Je_i, Je_i) \ ,
\end{equation*}
or
\begin{equation}
\frac{\partial \tilde{\theta}}{\partial t} = \sum _{i=1}^n \overline \nabla ^2\tilde{\theta} (e_i, e_i) \ .
\end{equation}

In the case that $n=1$, we actually have $e:=e_1=(\cos\tilde{\theta},\sin\tilde{\theta})$, so we can write the equation completely in terms of $\theta$. We summarize the result in the following

\begin{thm}
Let $\theta: W \longrightarrow \mathbb{R}/2\pi \mathbb Z$ corresponds to the Lagrangian angles for a foliation moved under mean curvature flow, where $W \subset \mathbb{R}^2 \times [0,T)$. Then $\theta$ satisties
\begin{equation}
\frac{\partial \theta}{\partial t} = \theta_{11}\cos ^2 \theta+ 2\theta_{12}\sin \theta \cos \theta + \theta_{22}\sin^2 \theta \ ,
\end{equation}
where $\theta _{ij} := \frac{\partial ^2 \theta}{\partial y^i \partial y^j}$ denotes the second partial derivative with respect to the standard coordinates $y^i$, $y^j$ in $\mathbb R ^2$.
\end{thm}

\section {Main Theorem}
Before going into the proof of the converse of {\bf Theorem 3.1},we define the following notations. Let $V_1$ be the time dependent vector field defined by $V_1=(\cos \theta, \sin \theta)$. We also write $V_2:= JV_1=(-\sin \theta, \cos \theta)$.

\begin{lem}
The equation (3.2), or (3.3), is equivalent to
\begin{equation}
\frac{\partial \theta}{\partial t} = \overline \nabla _{V_1}\overline \nabla _{V_1}\theta
- \overline \nabla _{V_1} \theta \overline \nabla _{V_2} \theta \ .
\end{equation}
\end{lem}
\begin{proof}
By (3.2)
\begin{equation*}
\begin{split}
\frac{\partial \theta}{\partial t} &= \overline \nabla ^2\tilde{\theta} (V_1, V_1)\\
&= \overline \nabla _{V_1}\overline \nabla _{V_1}\theta -\overline \nabla _{\overline\nabla _{V_1}V_1} \theta
\end{split}
\end{equation*}
As $V_1=(\cos \theta, \sin \theta)$
\begin{equation*}
\begin{split}
\overline \nabla _{V_1} V_1 &= \cos \theta \frac{\partial}{\partial x} (\cos \theta ,\sin \theta) + \sin \theta \frac{\partial}{\partial y} (\cos \theta, \sin \theta)\\
&= \theta_x \cos \theta (-\sin \theta, \cos \theta) + \theta_y \sin \theta (-\sin \theta, \cos \theta)\\
&= (\overline \nabla _{V_1}\theta) V_2
\end{split}
\end{equation*}
Thus we have $\frac{\partial \theta}{\partial t} = \overline \nabla _{V_1}\overline \nabla _{V_1}\theta - \overline \nabla _{V_1} \theta \overline \nabla _{V_2} \theta$.
\end{proof}

Now we can prove the main theorem:

\begin{thm}
Let $\theta : W \longrightarrow \mathbb R /2\pi \mathbb Z$ be a $\ci$ function satisfying equation (3.3), where $W \subset \mathbb{R}^2 \times [0,T)$. Then $\forall y\in W_0=\{(z,t)\in W: t=0 \}$, we can find $\{L_t: t \in [0,T_y) \}$, where $T_y \leq T$, such that
\begin{enumerate}
\item each $L_t$ is an integral curve of the vector field $V_1(\cdot,t)$, and $y\in L_0$
\item The family $L_t: t \in [0,T_y)$ is smooth in $t$ and moved by mean curvature flow.
\end{enumerate}
\end{thm}

\begin{proof}
Let $y\in W_0$. Then ODE theory implies that there is a curve $\alpha$ defined on $[0, T_y)$ for some $T_y \leq T$ such that
\begin{equation*}
\begin{cases}
\alpha '(t)=(\overline \nabla _{V_1}\theta) V_2(\alpha (t),t)\\
\alpha (0)=y \ \ \ \ \ \ \ \ \ \ \ \ \ \ \ \ \ \ \ .
\end{cases}
\end{equation*}
Using points on $\alpha$ as initial data, we find the integral curve of the vector field $V_1(\cdot,t)$ which pass through $\alpha(t)$. Standard result in ODE again states that the integral curves are smoothly dependent on initial data, and thus we have actually a $\ci$ map
$$X: (-\epsilon, \epsilon) \times [0,T_y) \longrightarrow \mathbb R ^2 \ ,$$
satisfying
\[
\begin{cases}
X(0,t)=\alpha (t)\\
X_s(s,t)=V_1(X(s,t),t)=(\cos \theta(X(s,t),t), \sin \theta(X(s,t),t)) \ .
\end{cases}
\]
$\forall (s,t) \in (-\epsilon, \epsilon) \times [0,T_y)$. Note that the curvature of the integral curve equals $\theta_s=\overline \nabla _{V_1}\theta$ and equation (4.1) can be written as
\begin{equation}
\frac{\partial \theta}{\partial t} = \theta _{ss} -(\overline \nabla _{V_2}\theta)\theta_s
\end{equation}

Let $t$ be fixed. We show that $\langle X_t, V_2\rangle = \theta_s$ for all $s\in (-\epsilon, \epsilon)$. To this end, we calculate
\begin{equation*}
\begin{split}
\frac{\partial}{\partial s}(\langle X_t, V_2 \rangle -\theta_s) &= \langle X_{ts}, V_2 \rangle + \langle X_t, (V_2)_s \rangle - \theta_{ss}\\
&= \langle X_{st}, V_2\rangle + \langle X_t, -\theta_sV_1\rangle - \theta_{ss} \ .\\
\end{split}
\end{equation*}
As $X_s = V_1 = (\cos \theta, \sin \theta)$,
\begin{equation*}
\begin{split}
X_{st} &= \theta_t (-\sin \theta, \cos \theta)\\
&= \theta_t V_2 \ ,
\end{split}
\end{equation*}
where $\theta_t := \frac{\partial}{\partial t} \theta(X(s,t),t)$. So we have
\begin{equation}
\frac{\partial}{\partial s}(\langle X_t, V_2 \rangle -\theta_s) = \theta _t -\theta _s \langle X_t, V_1\rangle - \theta_{ss} \ .
\end{equation}
Now we calculate $\theta _t$. As $\theta = \theta(X(s,t),t)$
\[
\begin{split}
\theta_t &= \langle \overline \nabla \theta ,X_t \rangle + \frac{\partial \theta}{\partial t}\\
&= \theta_s \langle V_1, X_t \rangle + \overline \nabla _{V_2}\theta \langle V_2, X_t \rangle + \frac{\partial \theta}{\partial t} \ ,
\end{split}
\]
where $\frac{\partial \theta}{\partial t}$ denotes the partial derivative of $\theta : W \longrightarrow \mathbb R /2\pi \mathbb Z$ with respect to the $t$ component. Putting it back into (4.3) and using (4.2), we have
\begin{equation*}
\begin{split}
\frac{\partial}{\partial s}(\langle X_t, V_2 \rangle -\theta_s) &= \theta_s \langle V_1, X_t \rangle + \overline \nabla _{V_2}\theta \langle X_t, V_2 \rangle + \frac{\partial \theta}{\partial t}-\theta_s \langle X_t, V_1 \rangle - \theta_{ss}\\
&= \overline \nabla _{V_2}\theta \langle X_t, V_2 \rangle + \frac{\partial \theta}{\partial t} - \theta_{ss}\\
&= \overline \nabla _{V_2}\theta(\langle X_t, V_2 \rangle -\theta_s) \ \ \ \ \ \ \ .
\end{split}
\end{equation*}
At $s=0$, $X(0,t)=\alpha(t)$. So $X_t(0,t)=\alpha '(t)=(\overline \nabla _{V_1}\theta) V_2$, and we have $\langle X_t,V_2 \rangle =\overline \nabla _{V_1}\theta = \theta_s$. Then by uniqueness of ODE we conclude that $\langle X_t, V_2 \rangle = \theta_s$ for all $s\in (-\epsilon, \epsilon)$.\\
Lastly, it is a standard result in theory of mean curvature flow that if $X_t^{\perp}=H$ is satisfied, by applying diffeomorphism on each $L_t:= X((-\epsilon, \epsilon), t)$, the integral curves are actually moved by mean curvature flow. This completes the proof of the theorem.
\end{proof}

\section{Examples of Invariant Solution}
In this section we consider examples in which $\theta$ is independent of $t$. The solutions are invariant in the sense that the all Lagrangian fibrations seem to be stationary. In this case equation (3.2), or (3.3), give
\begin{equation}
\overline \nabla ^2 \theta (e, e)=0 \ ,
\end{equation}
or
\begin{equation}
\theta_{11}\cos ^2 \theta+ 2\theta_{12}\sin \theta \cos \theta + \theta_{22}\sin^2 \theta = 0 \ .
\end{equation}

\begin{example}
Obviously $\theta(x,y)=ax+by+c$ is a solution of (5.2) for any constant $a$, $b$, $c$. In the trivial case where $a=b=0$, the plane is foliated by straight lines. Except this case we consider (after a rigid motion) $\theta(x,y)=ax$ for $a>0$. The result is the well known translation invariant solution called the $grim \ reaper$ \cite{CZ}. As a graph, it is given by
\begin{equation*}
v(x)=-\frac{1}{a} \ln \cos ax + C \ ,
\end{equation*}
where $C$ is a constant and $x\neq \frac{\pi}{2a}+\frac{m\pi}{a}$, for $m\in \mathbb Z$. When $x=\frac{\pi}{2a}+\frac{m\pi}{a}$, it is filled by straight lines. So in this case, $\mathbb C$ is foliated by $grim \ reapers$ and straight lines. This is also an example that under the mean curvature flow the Lagrangian fibers do not converge to special Lagrangian fibers, namely straight lines in $\mathbb C$.
\end{example}

\begin{example}
Consider $W=\mathbb R ^2 / \{0\}$ and use standard polar coordinate $(r,\theta)$. Define $\theta_c (r, \theta)= \theta + c$, where $c\in \mathbb R /2\pi \mathbb Z$ is a constant. Using the basis $\{ \frac{\partial}{\partial r}, \frac{\partial}{\partial \theta} \}$, the Hessian of $\theta_c$ is given by
$$\overline \nabla ^2 \theta _c =
\left( \begin{array}{cc} 0 & -\frac{1}{r} \\
-\frac{1}{r} & 0
\end{array} \right)$$
As a result, $\theta_c$ satisfies (5.1) if and only if $(\cos \theta_c , \sin \theta_c)$ is parallel to either $\frac{\partial}{\partial r}$ or $\frac{\partial}{\partial \theta}$. This is true if and only if $c=0, \pi$ or $c=\frac{\pi}{2}, \frac{3\pi}{2}$. The first case corresponds to fibration of $W$ by straight rays from the origin, while the second corresponds to fibration of $W$ by concentric circles centered at the origin.
\end{example}

Intuitively, every curves in the invariant solution should not have point where curvature is zero. As an application of result in section 4 we prove the

\begin{prop}
Let $\theta$ be a solution of (5.1). Assume $\overline \nabla _{V_1}\theta(y)=0$ for some $y\in W$. If $\alpha$ is the integral curve for the vector field $V_1=(\cos \theta, \sin \theta)$ which passes through $y$. Then $\alpha$ is a straight line.
\end{prop}

\begin{proof}
Let $\alpha : (-\epsilon, \epsilon)\longrightarrow W$ be the integral curve for the vector field $V_1$. Then
\begin{equation*}
\begin{cases}
\alpha '(s) = V_1(\alpha (s))\\
\alpha (0) = y \ \ \ \ \ \ \ \ \ \ .
\end{cases}
\end{equation*}
As $\overline \nabla _{V_1} \theta = \theta _s$, by (5.1) we have $\theta _{ss} = (\overline \nabla _{V_2}\theta) \theta_s $. Again by result in ODE, as $\theta_s(0)=\overline \nabla _{V_1}\theta (y)=0$ we have $\theta _s(s) = 0$ for all $s$. Thus $\theta(s) \equiv c$ and $\alpha $ is a straight line.
\end{proof}

As a consequence of {\bf Proposition 5.1}, if $\overline \nabla _{V_1} \theta = 0$ in an open set $\Omega \subset W$, then all integral curves passing through $\Omega$ are straight lines. In this case $\Omega$ is foliated by special Lagrangian fibers, namely straight lines in $\mathbb C$. The case $c=0$ in {\bf Example 5.2} represents such a situation with $\Omega = W$, as
\begin{equation*}
\begin{split}
V_1 &= (\cos \theta_0, \sin \theta_0)\\
&= (\cos \theta, \sin \theta)\\
&= \frac{\partial}{\partial r} \ ,
\end{split}
\end{equation*}
thus $\overline \nabla _{V_1} \theta_0 =\frac{\partial \theta}{\partial r} = 0$.

\bibliographystyle{amsplain}

\end{document}